\documentclass{amsart}
\usepackage{amssymb,latexsym, amsmath, amscd}
\usepackage{enumerate}
\usepackage{fullpage, booktabs}
\usepackage{graphicx}
\usepackage[all]{xy}
\makeatletter
\@namedef{subjclassname@2010}{%
\textup{2010} Mathematics Subject Classification}
\makeatother

\newtheorem{thm}{Theorem}

\newtheorem{lem}[thm]{Lemma}
\newtheorem{prop}[thm]{Proposition}

\theoremstyle{definition}

\newtheorem{rem}[thm]{Remark}

\numberwithin{equation}{section}

\newcommand\F[1]{F_{#1,3}(\mathbb{R})}

\def\Z{\mathbb Z}

\def\X{\mathcal X}

\def\PB{\overline{{B}}}
\def\PP{\overline{{P}}}

\begin{document}

\title[Planar pure braids on six strands]{planar pure braids on six strands}

\author[J. Mostovoy]{Jacob Mostovoy}
\address{Departamento de Matem\'aticas, CINVESTAV-IPN\\ Col. San Pedro Zacatenco, M\'exico, D.F., C.P.\ 07360\\ Mexico}
\email{jacob@math.cinvestav.mx}

\author[C. Roque-M\'arquez]{Christopher Roque-M\'arquez}
\address{Instituto de Matem\'aticas, UNAM\\
Antonio de Le\'on \#2, Altos, Col. Centro, Oaxaca, C.P. 68000\\ Mexico}
\email{croque@im.unam.mx}

\begin{abstract}
The group of planar (or flat) pure braids on $n$ strands, also known as the pure twin group, is the fundamental group of the configuration space $\F{n}$ of $n$ labelled points in $\mathbb{R}$ no three of which coincide. The planar pure braid groups on 3, 4 and 5 strands are free. In this note we describe the planar pure braid group on 6 strands: it is a free product of the free group on 71 generators and 20 copies of the free abelian group of rank two. 
\end{abstract}

%\subjclass[2010]{Primary }

%\keywords{}

\maketitle

\section{Introduction}

The group of planar braids on $n$ strands $\PB_n$ is the Coxeter group with the generators $\sigma_1, \ldots, \sigma_{n-1}$ and the relations
\[
\begin{array}{rcll}
\sigma_i^2&=& 1&\quad \text{for all\ } 1\leq i < n;\\
\sigma_i\sigma_j&=& \sigma_j\sigma_i &\quad \text{for all\ } 1\leq i, j <n \text{\ with\  } |i-j|>1.
\end{array}
\]

Elements of $\PB_n$ can be drawn as \emph{planar braids} on $n$ strands; these are collections of $n$ smooth descending arcs in $\mathbb{R}^2$ which join $n$ distinct points on some horizontal line with $n$ points directly below them on another horizontal line;  two arcs are allowed to intersect, although intersections of three or more arcs at the same point are forbidden. 
The arcs are considered up to a smooth isotopy of the plane which keeps the arcs descending and preserves the endpoints of the arcs, and modulo a local move similar to the second Reidemeister move on knots.

In particular, the generator $\sigma_i$ is represented by the following ``elementary'' braid:

\smallskip

$$
\includegraphics[width=60pt]{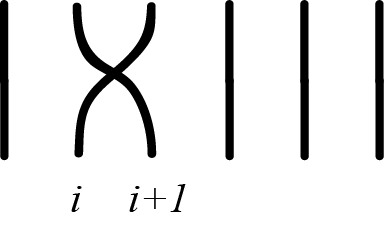}
$$

\smallskip

\noindent and the relations have the following form:

\smallskip

$$
\includegraphics[width=350pt]{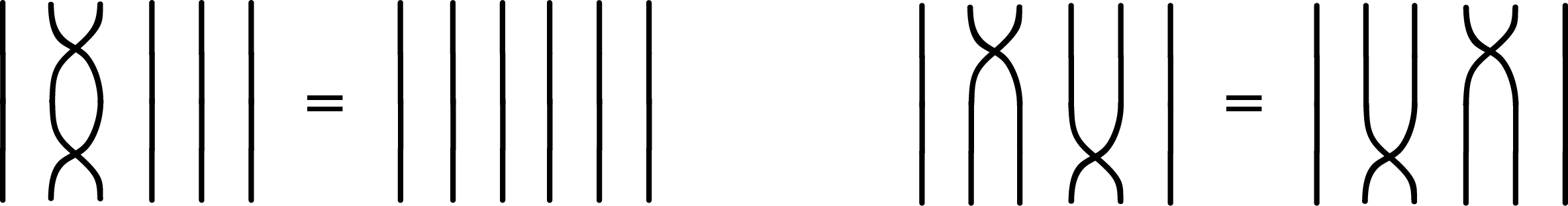}
$$

\smallskip

\noindent One obtains the symmetric group $S_n$ from $\PB_n$ by adding the relation
\[ (\sigma_i\sigma_{i+1})^3=1 \quad \text{for all\ } 1\leq i < n-1.\]
The quotient homomorphism $\PB_n\to S_n$ assigns to a planar braid the permutation obtained by following the strands of a braid; this is entirely similar to the case of the usual braids. The kernel of this homomorphism is the \emph{planar pure braid group} $\PP_n$; it consists of the planar braids each of whose strands ends directly below its initial point. 

The planar pure braid group is the fundamental group of the complement in $\mathbb{R}^n$ to the union of all the subspaces
\[x_i=x_j=x_k,\quad i<j<k.\]
This space, which we denote here by $\F{n}$, may be thought of as the configuration space of $n$ ordered particles in $\mathbb{R}$ no three of which are allowed to coincide.  A planar pure braid is simply a graph of a closed curve in this space. Khovanov showed in \cite{Kh1} that $\F{n}$ is a classifying space for $\PP_n$: its homotopy groups are trivial in dimensions greater than one.

\medskip 

The planar braid group appears under the names \emph{cartographical group} in \cite{V, ShV} and \emph{twin group} in \cite{Kh1, Kh2}. In \cite{Merk2} planar braids are called \emph{flat braids} (this term was later used in \cite{KL} to denote an entirely different object). The space $\F{n}$, also called the \emph{no-3-equal manifold of $\mathbb{R}$}, was studied in  \cite{B, BW, SW} and various other papers. 

The planar pure braid groups on 3, 4 or 5 strands are free. In this note we explicitly identify the planar pure braid group on six strands. 

\medskip

A few words about the conventions. We assume that for braids $a$ and $b$, the product $ab$ means ``$a$ on top of $b$''. We shall speak about the ``number of a strand'' for not necessarily pure braids: the strands will be counted by their upper endpoints from left to right. 

\section{The structure of $\PP_3$, $\PP_4$ and $\PP_5$}

\begin{prop}
The groups of pure planar braids on less than six strands are free:  $\PP_1$ and $\PP_2$ are trivial, $\PP_3=\Z$, $\PP_4 =F_7$ and $\PP_5=F_{31}$.
\end{prop}

\begin{proof}
It can be easily seen that $\F{1}$ and $\F{2}$ are contractible. The space $\F{3}$ is a complement to a line in 
$\mathbb{R}^3$ and, hence is, homotopy equivalent to a circle. A direct calculation shows that $\PP_3$ is generated by the following braid:
$$
\includegraphics[width=20pt]{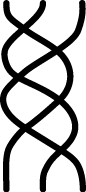}
$$

\smallskip

The space $\F{4}$ has a free action of $\mathbb{R}$ by translations so it retracts to its intersection with the hyperplane 
$x_1+x_2+x_3+x_4=0$ which is a complement in $\mathbb{R}^3$ to 4 lines passing through the origin; in turn, this space is homotopy equivalent to a  one-point union of 7 circles. The generators of $\PP_4$ are as follows:
\smallskip

$$
\includegraphics[width=250pt]{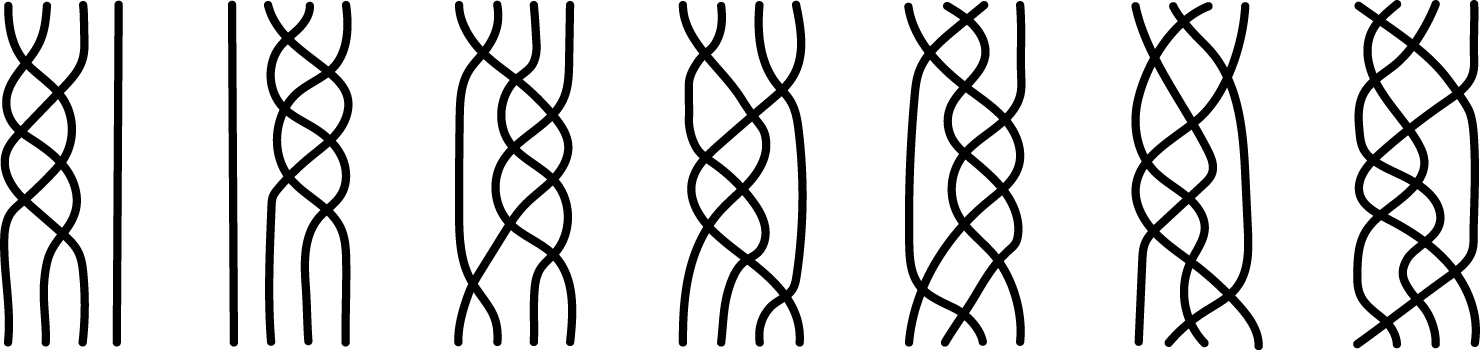}
$$

\smallskip
Finally, $\F{5}$, considered up to translations and positive rescalings is the complement in $S^3$ to a configuration of curves. This configuration is symmetric with respect to the centre of $S^3$; its intersection with one hemisphere can be thought of as the union of all the straight lines in a three-dimensional ball passing through certain five points:
$$
\includegraphics[width=120pt]{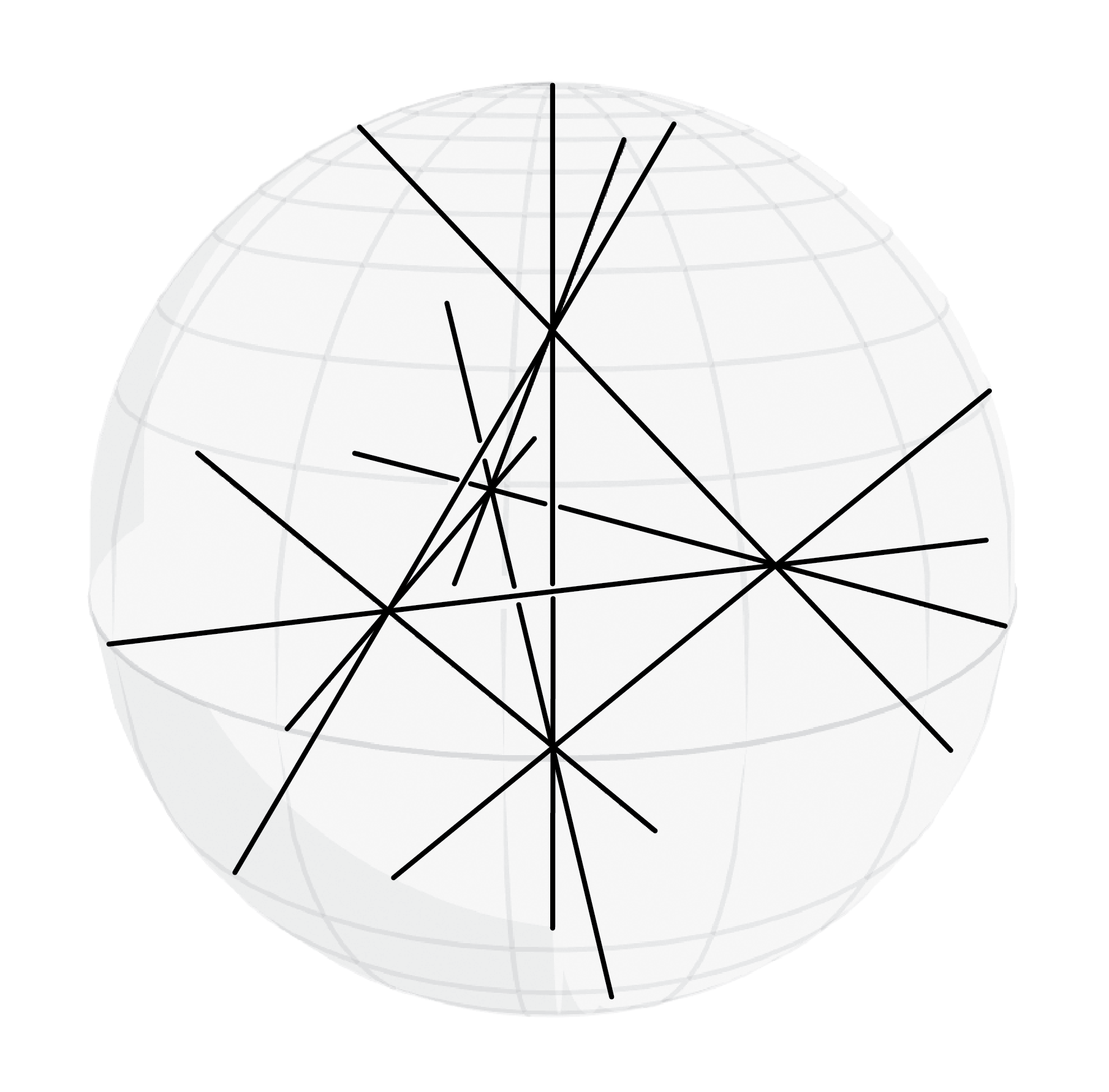}
$$
This can be seen to be the complement to the one-point union of 31 unlinked and unknotted  circles. See \cite{GGLMR} for a proof along different lines.
 \end{proof}

\section{Pure planar braids on 6 strands} 

\subsection{The structure of $\PP_6$}

It is known that for $n\geq 6$ the group $\PP_n$ is never free since the homology of $\F{n}$ in this case does not vanish in higher dimensions \cite{BW}.
\begin{thm}\label{six}
The group $\PP_6$ is isomorphic to the free product 71 copies of the infinite cyclic group and 20 copies of the free abelian group $F_2^{\mathrm{ab}}$ of rank two.
\end{thm}

Let us give an explicit description of the free factor $(F_2^{\mathrm{ab}})^{\ast 20}$. 

\medskip

There is an ``obvious'' pair of commuting elements in $\PP_6$, namely, the two braids $g$ and $h$ as in Figure~\ref{gandh}.
%\medskip
\begin{figure}[h]
$$\includegraphics[width=140pt]{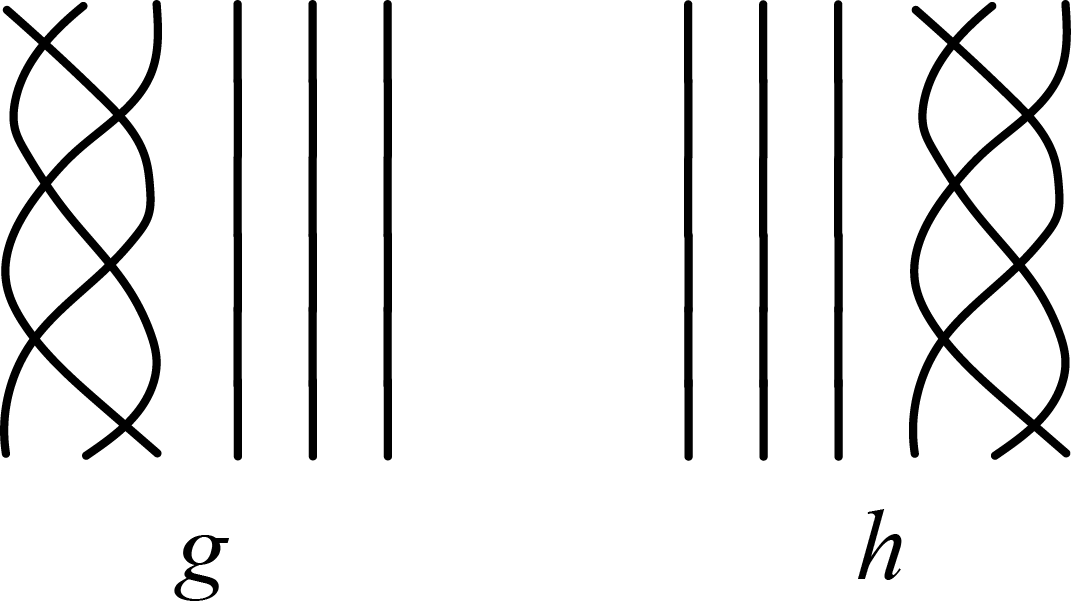}$$
\caption{}\label{gandh}
\end{figure}
%\medskip
For each integer triple $1\leq i_1< i_2 < i_3\leq 6$ define 
$q_{i_1 i_2 i_3}\in \PB_6$ as the planar braid whose strands are segments of straight lines and whose underlying permutation is the $(3,3)$-shuffle that sends $k$ to $i_k$ for $k=1,2,3$.  
The elements $q_{i_1 i_2 i_3}\in \PB_6$ are precisely those which can be drawn as braids whose first three strands, as well as the last three strands, do not intersect each other. For instance, here is the braid $q_{136}$:
\smallskip

$$\includegraphics[width=50pt]{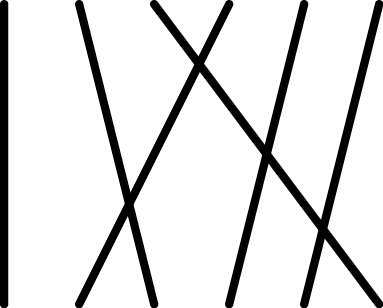}$$

\smallskip

There are exactly 20 triples $1\leq i_1< i_2 < i_3\leq 6$. 
Under the isomorphism $\PP_6=F_{71}\ast (F_2^{\mathrm{ab}})^{\ast 20}$ that we exhibit in the proof of Theorem~\ref{six}, the 20 pairs of braids $(q_{i_1 i_2 i_3}^{-1} \cdot g\cdot q_{i_1 i_2 i_3}, q_{i_1  i_2  i_3}^{-1}\cdot h\cdot q_{i_1  i_2  i_3})$ generate the 20 copies of $F_2^{\mathrm{ab}}$.

\medskip

Theorem~\ref{six} was initially established with the help of a computer implementation of the Reidemeister-Schreier method (as applied in \cite{H} to the case of the usual pure braids), followed by some \emph{ad hoc} fiddling. The proof presented here also follows the Reidemeister-Schreier type argument.  We consider the sequence of subgroups $\PP_6\subset H\subset \PB_6$ where $H$ consists of the planar braids whose underlying permutation sends the set $\{1,2,3\}$ to itself (and, therefore, maps the set $\{4,5,6\}$ to itself as well). The group $H$ can be thought of as the group of \emph{bicoloured} braids, say, with the first three strands red and the last three black\footnote{in a nod to \cite{Leznov}.}. It turns out that $H$ is a free product of two copies of the same group which we describe explicitly; this, in turn, produces a description of $\PP_6$. We should mention that the planar pure braids have been studied using the Reidemeister-Schreier method before; see \cite{BSV}.
 
\subsection{The structure of $H$}

\begin{figure}[ht]
\includegraphics[width=300pt]{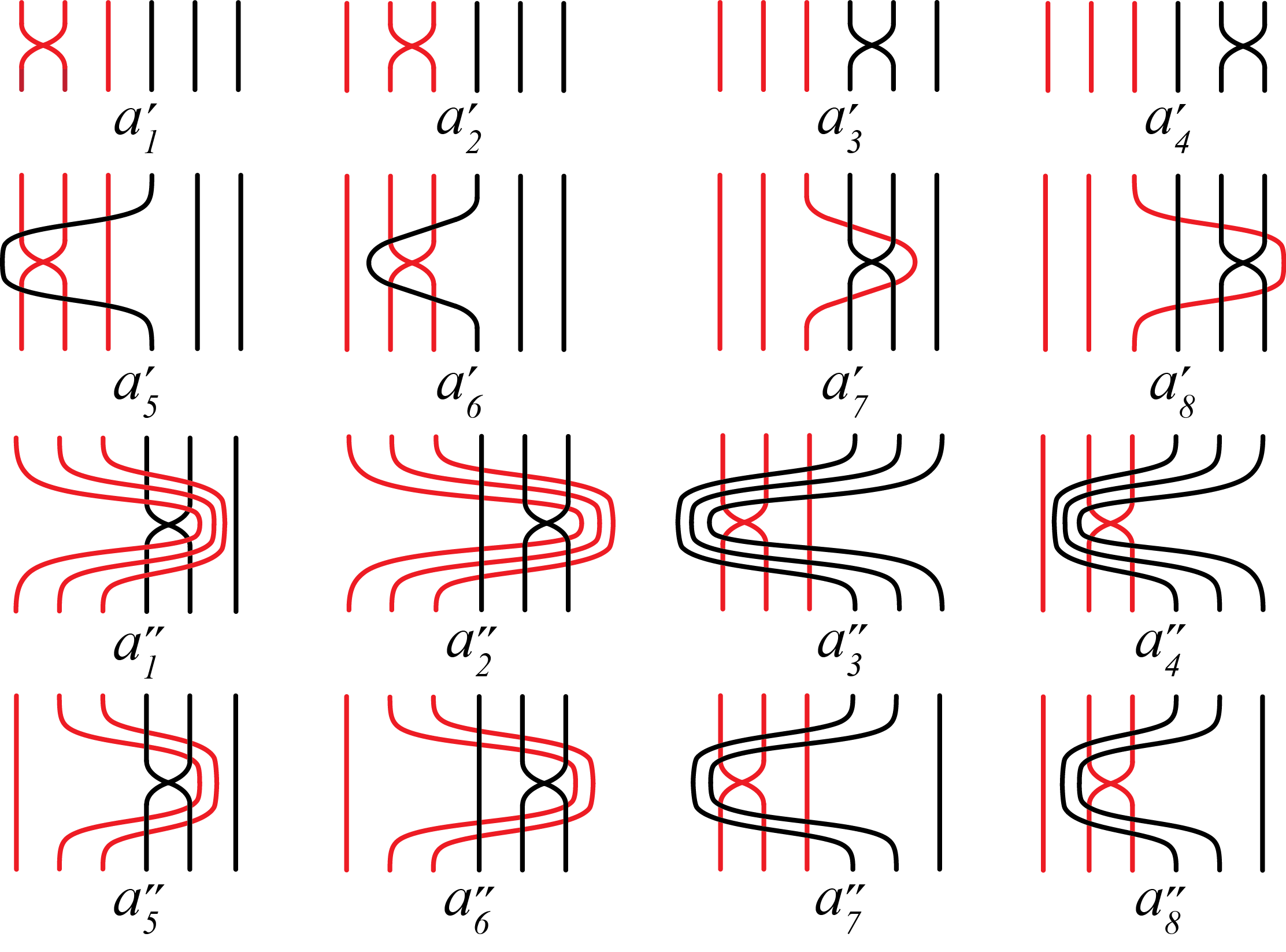}
\caption{The generators of $H$.}\label{Hgen}
\end{figure}

Consider the group $A$ generated by $a_1,\ldots, a_{8}$ subject to two kinds of relations: 
\begin{itemize}
\item $a_i^2=1$ for all $i$; 
\item $a_ia_j=a_ja_i$ whenever the vertices $i$ and $j$ on the graph in Figure~\ref{product} are connected by an edge.  
\end{itemize}
\begin{rem}
A partially commutative group defined in this fashion is called a \emph{graph product}; in this case, of 8 groups of order two, see \cite{Green}.
\end{rem}
\begin{figure}[hb]
\includegraphics[width=100pt]{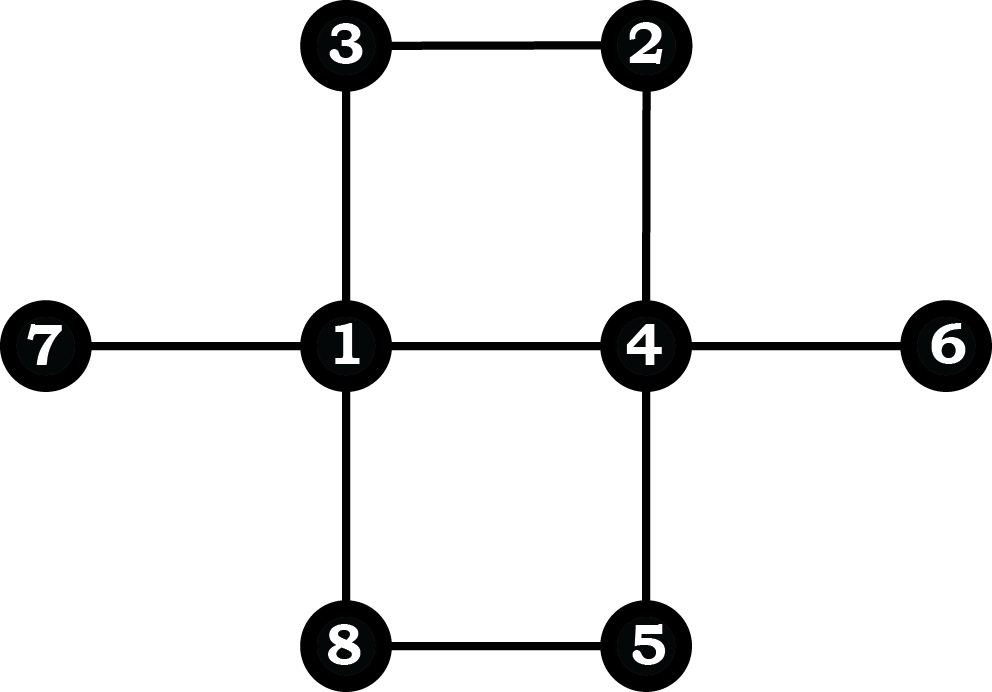}
\caption{The graph of relations in $A$.}\label{product}
\end{figure}

\medskip

Let $A'$ and $A''$ be two copies of $A$ with the generators $a'_1,\ldots a'_8$ of $A'$ and $a''_1,\ldots a''_8$ of $A''$ corresponding to the generators $a_1,\ldots, a_8$ of $A$.

\begin{prop}\label{H}
The group of bicoloured braids $H$ is isomorphic to the free product $A'\ast A''$. The bicoloured braids which correspond to the generators $a'_1,\ldots a'_8$ and $a''_1,\ldots, a''_8$  are shown in Figure~\ref{Hgen}.
\end{prop}
 
In particular, conjugating $A'$ by $q_{456}$ inside $H$ we obtain $A''$. 

\begin{proof}

Assigning to each $a_i$ a planar braid as in Figure~\ref{Hgen} we get a homomorphism of the group described in the statement of the Proposition to $H$. Let us construct its inverse.

First, given $b=\sigma_{j_1}\ldots\sigma_{j_m}\in H$ we show how to write $b$ as the product of the $a'_i$ and $a''_j$. Write $s_x$ for the permutation defined by $x\in \PB_6$ and set $q_k=q_{i_1 i_2 i_3}$ where 
$$\{i_1, i_2, i_3\} = \{s_{\sigma_{j_1}\ldots\sigma_{j_k}}(1), s_{\sigma_{j_1}\ldots\sigma_{j_k}}(2), s_{\sigma_{j_1}\ldots\sigma_{j_k}}(3)\}.$$
Then, $$b=(q_0\sigma_{j_1} q_1^{-1})(q_1 \sigma_{j_2}q_2^{-1})\ldots (q_{m-1} \sigma_{j_m}q_m^{-1}),$$
since $q_0=q_m=1$.

For each of the factors $q_{k-1} \sigma_{j_k}q_k^{-1}$ there are two possibilities: $\sigma_{j_k}$  may either involve two strands of different colours or two strands of the same colour. If the colours of the strands are different, we have $q_{k-1}\neq q_k$ and $q_{k-1} \sigma_{j_k}q_k^{-1}=1$. Indeed, in this case, if $q_{k-1}$ can be represented by a braid whose first three, 
as well as the last three, strands are disjoint, the same is true for $q_{k-1} \sigma_{j_k}$; this implies $q_{k-1} \sigma_{j_k} = q_k$.

On the other hand, when the factor $\sigma_{j_k}$ involves a crossing of two strands of the same colour, we have $q_{k-1}=q_k$. Aside from this crossing, the first three strands, as well the last three strands of the braid $q_{k} \sigma_{j_k}q_k^{-1}$ do not cross  each other. Figure~\ref{Hgen} lists all such braids; these are the generators $a'_i$ and $a''_j$. Indeed, Table~\ref{genH} shows all the pairs $(q_{i_1i_2i_3}, \sigma_m)$ such that  the factor $\sigma_{m}$  in the braid $q_{i_1i_2i_3}\sigma_{m} q_{i_1i_2i_3} ^{-1}$has both strands of the same colour. 

\begin{table}[h!]
  \begin{center}
    \caption{Generators and relations for $H$. The intersection of the row $i_1i_2i_3$ with the column $\sigma_k$ shows the  non-trivial values of $q_{i_1i_2i_3}\cdot\sigma_k\cdot q_{i_1i_2i_3}^{-1}$.}
    \label{genH}
    \begin{tabular}{l|ccccc|l}
           & $\sigma_1$ & $\sigma_2$ & $\sigma_3$ & $\sigma_4$ & $\sigma_5$&commutation relations\\  
\hline
123     &$a'_1$ & $a'_2$ &  & $a'_3$ &$a'_4$& $[a'_1,a'_3], [a'_1,a'_4], [a'_2, a'_3], [a'_2,a'_4]$ \\
124     &$a'_1$ &  &  &  &$a'_4$& $[a'_1,a'_4]$\\ 
125     & $a'_1$ &  & $a'_7$ &  && $[a'_{1}, a'_{7}]$\\ 
126     & $a'_1$ &  & $a'_7$ & $a'_8$ &&$[a'_{1}, a'_{7}],[a'_{1}, a'_{8}]$\\
134     &  &  & $a'_6$ &  & $a'_4$&$[a'_{4}, a'_{6}]$\\
135     &  &  &  &  &   &\\
136     &  &  &  & $a'_8$ & &\\
236 &  & $a'_{5}$ &  & $a'_{8}$ & &$[a'_{5}, a'_{8}]$\\
235 &  & $a'_{5}$ &  &  & &\\
234 &  & $a'_{5}$ & $a'_{6}$ &  & $a'_{4}$&$[a'_{4}, a'_{5}],[a'_{4}, a'_{6}]$\\
456     & $a''_{1}$  & $a''_{2}$ &  & $a''_{3}$ & $a''_{4}$&$[a''_{1}, a''_{3}], [a''_{2}, a''_{3}], [a''_{1}, a''_{4}], [a''_{2},a''_{4}]$\\
356     & $a''_{1}$ &  &  &  & $a''_{4}$ & $[a''_{1}, a''_{4}]$\\
346 & $a''_{1}$ &  & $a''_{7}$ &  & &$[a''_{1}, a''_{7}]$\\
345 & $a''_{1}$ &  & $a''_{7}$ & $a''_{8}$ &&$[a''_{1}, a''_{7}],[a''_{1}, a''_{8}]$ \\
256 &  & & $a''_{6}$ &  & $a''_{4}$ &$[a''_{4}, a''_{6}]$ \\
246 &  &  &  &  &&\\
245 &  &  &  & $a''_{8}$ & &\\
145     &  & $a''_{5}$ &  & $a''_{8}$ && $[a''_{5}, a''_{8}]$\\
146     &  & $a''_{5}$ &  &  & &\\
156     &  & $a''_{5}$ & $a''_{6}$ &  & $a''_{4}$ &$[a''_{4}, a''_{5}], [a''_{4}, a''_{6}]$
    \end{tabular}
  \end{center}
\end{table}

Now, assume that $b'$ differs from $b$ by applying one of the relations in $\PB_6$. If $b'$ is obtained from $b$ by 
exchanging $\sigma_{j_k}$ with $\sigma_{j_{k+1}}$ with $|j_{k+1}-j_k|>1$ we have
$$b = \ldots (q_{k-1}\sigma_{j_k} q_k^{-1})(q_k \sigma_{j_{k+1}}q_{k+1}^{-1})\ldots $$
and
$$b' = \ldots (q_{k-1}\sigma_{j_{k+1}} {q_k'}^{-1})(q_k' \sigma_{j_{k}}q_{k+1}^{-1})\ldots.$$
Then, $b$ and $b'$, as words in the $a'_i$ and $a''_j$, differ by a non-trivial relation only if $q_{k-1}=q_k=q_k'=q_{k+1}$. Such a relation states that two of the generators $a'_i$ or $a''_j$ commute: inspection shows that these relations are the same for $a'_i$ and for $a''_j$ (never involving generators of both kinds) and coincide with the commutation relations in $A$; see Table~\ref{genH}. In the same fashion, the relations $\sigma_{j}^2=1$ produce the relations ${a'_i}^2={a''_j}^2=1$.
\end{proof}

\medskip

\subsection{The kernel of a free product of morphisms}

\begin{lem}\label{freeproduct}
Let $f': G'\to S$ and $f'': G''\to S$ be surjective homomorphisms of groups with kernels $K'$ and $K''$ respectively. 
Then, the kernel of the induced homomorphism $f'*f'':G'\ast G''\to S$ is isomorphic to $K'\ast K''\ast F_{|S|-1}$.
\end{lem}
\begin{proof}
Choose representatives $z'_\alpha\in G'$ and  $z''_\alpha\in G''$  for each element $\alpha\in S$ assuming that $e\in S$ is represented by the neutral elements in $G'$ and $G''$.

Consider the free group $F_{|S|-1}$ whose generators $x_\alpha$ are indexed by $\alpha\in S-\{e\}$. We have a homomorphism $F_{|S|-1}\to G'\ast G''$ which sends $x_\alpha$ to $z'_\alpha {z''}^{-1}_\alpha$. Together with the natural inclusions $K'\hookrightarrow G'$ and $K''\hookrightarrow G''$ it gives a homomorphism 
$$\phi: K'\ast K''\ast F_{|S|-1}\to \ker{(f'\ast f'')}\subset G'\ast G''.$$

Any $y\in \ker{(f\ast f)}$ can be written as $y=r_1 s_1 \ldots r_m s_m$ with $r_i\in G'$ and $s_j\in G''$. Write 
$$\alpha_k = (f'*f'')(r_1 s_1 \ldots r_k)$$
and
$$\beta_k = (f'*f'')(r_1 s_1 \ldots r_k s_k).$$
Then we have
%\begin{multline*} 
$$
y=  \bigl(r_1  {z'}^{-1}_{\alpha_1}\bigr) \cdot  
\bigl( z'_{\alpha_1} {z''}^{-1}_{\alpha_1} \bigr) \cdot  
\bigl( z''_{\alpha_1} s_1  {z''}^{-1}_{\beta_1}\bigr)  \cdot 
 \bigl( {z''}^{-1}_{\beta_1} z'_{\beta_1} \bigr) \cdot 
%\ldots \\
\ldots\cdot \bigl(z'_{\beta_{m-1}} r_m  {z'}^{-1}_{\alpha_m}\bigr) \cdot  
\bigl( z'_{\alpha_m} {z''}^{-1}_{\alpha_m}\bigr) \cdot  
\bigl( z''_{\alpha_m} s_m \bigr) 
$$
%\end{multline*}
In this expression, $(z'_{\beta_{k-1}} r_k  {z'}^{-1}_{\alpha_k})\in K'$ while $(z''_{\alpha_k} s_k  {z''}^{-1}_{\beta_k}) \in K''$. Moreover, $(z'_{\alpha_k} {z''}^{-1}_{\alpha_k})$, if non-trivial, is the image of the generator $x_{\alpha_k}\in F_{|S|-1}$ and, therefore, replacing $(z'_{\alpha_k} {z''}^{-1}_{\alpha_k})$ with $x_{\alpha_k}$ (and $({z''}^{-1}_{\beta_k} z'_{\beta_k})$ with $x^{-1}_{\beta_k}$), we obtain an element $$\tilde{y}\in K'\ast K''\ast F_{|S|-1}$$ such that $\phi(\tilde{y})=y$. Call this element the \emph{rewriting} of the product $r_1 s_1 \ldots r_m s_m$.

We claim that the rewriting $\tilde{y}$ depends only on $y$. Indeed, replace $s_k$ in the product $r_1 s_1 \ldots r_m s_m$ by 
$s_{k_1} e s_{k_2}$ where $s_{k_1}s_{k_2}=s_k$ and $e$ is considered as an element of  $G'$. In the rewriting, this will result in replacing 
$(z''_{\alpha_k} s_k  {z''}^{-1}_{\beta_k})$ with 
$\bigl(z''_{\alpha_k} s_{k_1}{z''}^{-1}_{\gamma} \bigr) \cdot  \bigl(z''_{\gamma} s_{k_2}{z''}^{-1}_{\beta_{k}} \bigr)$
where $$\gamma=(f'\ast f'')(r_1s_1\ldots r_k s_{k_1});$$
this does not alter the rewriting as an element of $K'\ast K''\ast F_{|S|-1}$. The same is true if $r_k$ is replaced by $r_{k_1} e r_{k_2}$ with $r_{k_1}r_{k_2}=r_k$.

It remains to observe that for $w\in K'\ast K''\ast F_{|S|-1}$, the rewriting of $\phi(w)$ coincides with $w$. This can easily be established by induction on the length of $w$ considered as a word in $t'\in K'$, $t''\in K''$ and the $x_\alpha$. Therefore, the rewriting gives a two-sided inverse to $\phi$ and the kernel of $f'\ast f''$ is isomorphic to $K'\ast K''\ast F_{|S|-1}$.
\end{proof}

Lemma~\ref{freeproduct} also has an elegant topological proof, pointed out to the authors by Omar Antol\'\i n Camarena; we sketch it here very briefly.

Replace the groups in question by their classifying spaces. Consider the pushout square
$$\begin{array}{ccc}
*&\to&BG''\\
\downarrow&&\downarrow\\
BG'&\to& BG'\vee BG''
\end{array}$$
Each of its vertices maps to $BS$ by means of the trivial map, $B{f'}$, $B{f''}$ and $B({f'\ast f''})$, respectively. The homotopy fibres of these maps also form a pushout square
$$\begin{array}{ccc}
\Omega BS=S&\to&BK''\\
\downarrow&&\downarrow\\
BK'&\to& B(\ker{{f'\ast f''}})
\end{array}$$
Therefore, $B(\ker{{f'\ast f''}})$ is homotopy equivalent to the $BK'$ connected to $BK''$ by $|S|$ intervals; that is, to $BK'\vee BK''\vee (S^1)^{|S|-1}$.

\subsection{Proof of Theorem~\ref{six}}

Consider the homomorphism $H \to S_3\times S_3$ which sends a braid into its underlying permutation; its kernel is $\PP_6$. Since  $H=A\ast A$ and $S_3\times S_3$ has 36 elements, Theorem~\ref{six} will be proved as soon as we establish the following:

\begin{prop}\label{clave}
The kernel $Q$ of the map $A\to S_3\times S_3$ which sends the generator $a_i$ into the permutation of the braid $a'_i$ is isomorphic to the free product of $F_{18}$ and 10 copies of $F_2^{\mathrm{ab}}$.
\end{prop}

The rest of this section is dedicated to the proof of this statement.

\medskip

The cosets of $Q$ in $A$ are pairs of permutations on 3 letters. The set $Z_3\times Z'_3$ where
$$Z_3:=\{1, a_1, a_2, a_1 a_2, a_2 a_1, a_1 a_2 a_1\}$$
and 
$$Z'_3:=\{1, a_3, a_4, a_3 a_4, a_4 a_3, a_4 a_3 a_4\}$$
is a system of representatives for these cosets.

Consider $b=a_{j_1}\ldots a_{j_m}\in Q$ and let $\mu_k\in Z_3\times Z'_3\subset A$ be the representative of the coset of $a_{j_1}\ldots a_{j_k}\in A$, with $\mu_0=1$. Then, we have
$$b=(\mu_0 a_{j_1} \mu_1^{-1})\ldots (\mu_{m-1} a_{j_m} \mu_m^{-1}).$$
Let $\X\subset Q$ be the subset consisting of the non-trivial elements of the form $\mu a_{j} \nu^{-1}$ with $\mu,\nu\in Z_3\times Z'_3$ and $0<j\leq 8$.

\begin{lem} If $\mu a_{j} \nu^{-1}=\mu' a_{k} \nu'^{-1}\in \X$, then $j=k$.
\end{lem}
As a corollary, $\X=\sqcup\, \X_j$ where $\X_j$ consists of the elements of the form $\mu a_{j} \nu^{-1}$.
\begin{proof}
First, note that $\X_1=\X_4=\emptyset$. Indeed, for each $\mu\in Z_3\times Z'_3$ we have that $\mu a_1\in Z_3\times Z'_3$ and thus $\mu a_1 \nu^{-1}\in Q$ with $\nu \in Z_3\times Z'_3$ implies $\mu a_1 \nu^{-1}=1$. The same is true for $a_4$ instead of $a_1$.

In order to see that the sets $\X_j$ are disjoint we observe that the index $j$ can be recovered from the intersection properties of the braids representing the elements $\mu a_{j} \nu^{-1}\in \X$. Indeed, $\mu a_{j} \nu^{-1}\in \X$ can be drawn as a braid whose black strands do not intersect each other if and only if  $j\in\{2,5,6\}$. 

When this happens, we can distinguish $\X_2$ as consisting of the elements which can be drawn so that the black strands do not intersect the red strands and $\X_5$ consists of the elements which can only be drawn in such a way that each red strand intersects some black strand.

Similarly, when $j\in\{3,7,8\}$, all the red strands are disjoint from each other and the same argument applies.

\end{proof}

The group $Q$ is generated inside $A$ by the elements of $\X$ subject to two kinds of relations.
The relations of the first kind come from the relations $a_j^2=1$ and are of the form $$(\mu a_j \nu^{-1}) (\nu a_j \mu^{-1})=1,$$
where $\mu$ and $\nu$ satisfy $\mu a_j \nu^{-1}\in \PP_6$. 
In particular, the set $\X_j$ contains, together with each element, its inverse.

The relations of the second kind come from the relations $a_j a_k = a_k a_j$ shown in the graph on Figure~\ref{product};
these are of the form 
$$(\lambda a_j \mu^{-1}) (\mu a_k \nu^{-1})=(\lambda a_k \tilde{\mu}^{-1}) (\tilde{\mu} a_j \nu^{-1}).$$

Recall that the set $\X_1$ is trivial; the relations that involve $\mu a_1\nu^{-1}$ are all of the form 
$$\mu a_i\nu^{-1} = (\mu\sigma_1) a_i(\nu\sigma_1)^{-1},$$
where $i\in\{3,7,8\}$. Recall that we also have $\mu a_i\nu^{-1} = (\nu a_i\mu^{-1})^{-1}$.
When $i=7$, there are no other relations involving the elements of $\X_7$ and the four elements $\mu, \mu\sigma_1, \nu$ and $\nu\sigma_1$ are all different; therefore, $|Z_3\times Z'_3|/4=9$ and $\X_7$ contributes $|Z_3\times Z'_3|/4=9$ free generators. Another 9 free generators come from the set $\X_6$ which consists of braids that are mirror reflections of those in $\X_7$.

Similarly, one sees that $\X_8$ and $\X_5$ produce 9 generators each. For each of the generators coming from $\X_8$, there is one generator coming from $\X_5$ with which it commutes; there are no other relations involving them. Therefore, we see that $\X_8$ and $\X_5$ contribute a free factor isomorphic to $(F_2^{\mathrm{ab}})^{*9}$. Computation shows that the commuting pairs $(\mu a_5 \nu^{-1}, \mu a_8 \nu^{-1})$ are 
all of the form $(q\cdot g \cdot q^{-1}, q \cdot h \cdot q^{-1})$ where $q=q_{i_1i_2i_3}$ and $(i_1, i_2, i_3)$ varies over the set of triples $$(1,2,4),\,(1,2,5),\,(1,2,6),\,(1,3,4),\,(1,3,5),\,(1,3,6),\,(2,3,4)\,(2,3,5),\,(2,3,6),$$
and $g$ and $h$ are shown on Figure~\ref{gandh}.

When $i=3$, the set $\X_3$ consists of two mutually inverse elements, namely, the braid $h$ and its inverse. Similarly, the set $\X_2$ consists of the braid $g$ and its inverse. The only relation that these braids satisfy states that they commute; this gives another free factor of the form $F_2^{\mathrm{ab}}$. This exhausts the list of generators and Proposition~\ref{clave} is proved.

\medskip

Note that the commuting pairs of generators of $\PP_6$ that lie in $A''\subset H$ are the conjugations by $q_{456}$ of the commuting pairs in $A'\subset H$. This shows that the commuting pairs of generators in $\PP_6$ can be chosen to be of the form $(q\cdot g \cdot q^{-1}, q \cdot h \cdot q^{-1})$ where $q=q_{i_1i_2i_3}$ for all possible $1\leq i_1< i_2< i_3\leq 6$.

\subsection*{Acknowledgments}
We would like to thank Omar Antol\'\i n Camarena, Fred Cohen and Jes\'us Gonz\'alez for useful conversations. The authors would like to thank for hospitality the Higher School of Economics (J.M.) and the Samuel Gitler Collaboration Center (C.R.-M.).

\end{document}